\NewCommandCopy
{\latexref}{\ref}
\NewCommandCopy{\latexpageref}{\pageref}
\documentclass{birkjour}
\RenewDocumentCommand{\ref}{m}{{\rm\latexref{#1}}}
\RenewDocumentCommand{\pageref}{m}{{\upshape\latexpageref{#1}}}
 \newtheorem{theorem}{Theorem}[section]
 
 \newtheorem{lemma}[theorem]{Lemma}
 \newtheorem{proposition}[theorem]{Proposition}
 \theoremstyle{definition}
 
 \theoremstyle{remark}
 \newtheorem{remark}[theorem]{Remark}
 
 \numberwithin{equation}{section}

\begin{document}

%
%
%
%
%
%
%
%
%

\title[Critical point localization and multiplicity results]{Critical point
localization and multiplicity results in Banach spaces via Nehari manifold
technique}

\author[R. Precup]{Radu Precup}
\address{%
Faculty of Mathematics and Computer Science and Institute of Advanced Studies in Science and Technology\\
Babe\c{s}-Bolyai University, 400084 Cluj-Napoca, Romania\\
\& Tiberiu Popoviciu Institute of Numerical Analysis, Romanian Academy\\
P.O. Box 68-1, 400110 Cluj-Napoca, Romania}
\email{radu.precup@ubbcluj.ro}
%
\author[A. Stan]{Andrei Stan}
\address{%
Faculty of Mathematics and Computer Science, Babe\c{s}-Bolyai University, \\400084 Cluj-Napoca, Romania\\
\& Tiberiu Popoviciu Institute of Numerical Analysis, Romanian Academy\\
P.O. Box 68-1, 400110 Cluj-Napoca, Romania}
\email{andrei.stan@ubbcluj.ro}





\begin{abstract}
In the paper, results on the existence of critical points in annular subsets
of a cone are obtained with the additional goal of obtaining multiplicity
results. Compared to other approaches in the literature based on the use of
Krasnoselskii's compression-extension theorem or topological index methods,
our approach uses the Nehari manifold technique in a surprising combination
with the cone version of Birkhoff-Kellogg's invariant-direction theorem.
This yields a simpler alternative to traditional methods involving
deformation arguments or Ekeland's variational principle. The new method is
illustrated on a boundary value problem for $p$-Laplace equations, and we
believe that it will be useful for proving the existence, localization, and
multiplicity of solutions for other classes of problems with variational
structure.
\end{abstract}

\maketitle
\section{Introduction}

Finding localized solutions of equations (often equivalent to finding
critical points of a given functional) in predefined domains is of interest
in mathematical models, as it provides a certain degree of control over the
solutions of the modeled system $-$ for example, one may seek a solution
whose energy remains within specific bounds. However, this approach
introduces additional challenges compared to critical point theory in the
entire space, primarily due to the behavior of the functional at the
boundaries. For instance, if a functional attains its minimum at a boundary
point, that point is not necessarily a critical point in the usual sense, as
the directional derivatives may not vanish in every direction of the space.

Some of the earliest attempts to localize critical points date back to
Schechter \cite{s1,s2} (see also \cite{st}). Using pseudogradients and
deformation arguments and imposing a boundary condition on the sphere,
Schechter established the existence of critical points within a ball under a
compactness Palais-Smale condition on the functional. It can be said that
Schechter's theorems (for minimizers and points of mountain pass type) are
critical point versions of Schaefer's fixed point theorem in a ball \cite[p.
139]{ag}, a particular case of the general Leray-Schauder fixed point
theorem \cite[Theorem 6.5.4]{ag}. Since 2008 \cite{p0}, the first author has
been interested in locating critical points in annular subsets of a cone,
with the adjacent goal of obtaining multiple solutions in such disjoint
sets. Critical point results in annular conical sets can be seen as
extensions of Krasnoselskii's fixed point theorem for cones. Similar to
Schechter's approach, the methods in \cite{p0} relied on deformation
arguments within Hilbert spaces, exploiting their rich geometric structure.
Later, in papers \cite{lpv} and \cite{ppv}, analogous results were obtained
in Banach spaces with some geometric properties. The alternative method of
obtaining critical points, based on Ekeland's variational principle, has
also been used for localization in bounded conical sets \cite{p1}, \cite{ppv}%
.

An interesting idea is to search for critical points on specific subsets
where they are likely to lie. A classical example is the Nehari manifold,
which has been extensively studied in the literature. A particularly
insightful and comprehensive reference on this topic is the paper by Szulkin
and Weth \cite{sw}; see also \cite{cir}, \cite{cd}, \cite{dp}, \cite{s} and
\cite{w}. For a real Banach space $X$ and a $C^{1}$ functional $E$, the
corresponding Nehari manifold is defined as
\begin{equation*}
\mathcal{N}:=\{u\in X\setminus \left\{ 0\right\} \,:\ \langle E^{\prime
}(u),u\rangle \ =0\},
\end{equation*}%
where $\langle \cdot ,\cdot \rangle $ denotes the dual pairing between $%
X^{\ast }$ and $X$. Obviously, any nonzero critical point of $E,$ i.e.,
solution of the equation $E^{\prime }\left( u\right) =0,$ belongs to $%
\mathcal{N}.$ It may happen that the converse statement is also valid for
certain points of $\mathcal{N}$ with appropriate properties. For example, we
can look for points that minimize the functional $E$ on $\mathcal{N}$, even
if $E$ is unbounded from below on the entire domain.

Given the parallelism that can be highlighted between the fixed point and
critical point theory, it is natural to assume that a deeper interaction of
the two theories would be possible. This is exactly the purpose of this
work, which for the first time combines the Nehari manifold technique with
the topological invariant-direction theorem of Birkhoff and Kellogg, thus
obtaining results for locating solutions in annular conical sets without
using Ekeland's principle. The idea is to use a cone version of the
Birkhoff-Kellogg theorem for a given operator defined on a domain whose
boundary coincides with the Nehari manifold, to guarantee the existence of
an eigenvalue and an eigenvector. Then, using the definition of the Nehari
manifold, it is shown that the eigenvalue must be equal to one, which
implies that the corresponding direction is a critical point of the
functional. Since the invariant-direction theorem is fundamentally derived
using the fixed point index, our approach effectively combines critical
point techniques (the Nehari manifold method) with fixed point methods.

In \cite{as}, the second author exploits the method of the Nehari manifold,
combining it with Ekeland's variational principle to obtain solutions within
annular domains. The present work aims not only to demonstrate a natural and
somewhat unexpected application of the Birkhoff-Kellogg theorem but also to
extend and strengthen the results of \cite{as} in several key directions:
first and foremost, we generalize the theory from Hilbert spaces to Banach
spaces. Secondly, we relax the regularity assumption on the functional,
requiring only $C^{1}$ smoothness instead of $C^{2}$, and finally, some
conditions imposed in \cite{as} are no longer necessary in our framework.

\section{Preliminaries}

In this section we recall some basic notions and results that are used
throughout the paper.

\subsection{The duality mapping}

Let $X$ be a real Banach space with norm $\left\vert \cdot \right\vert $, $%
X^{\ast }$ its dual space, and let $\langle \cdot ,\cdot \rangle $ denote
the dual pairing between $X^{\ast }$ and $X$. A function $\varphi \colon
\mathbb{R}_{+}\rightarrow \mathbb{R}_{+}$ is said to be a normalization
function if it is continuous, strictly increasing, $\varphi \left( 0\right)
=0$ and $\varphi \left( \tau \right) \rightarrow \infty $ as $\tau
\rightarrow \infty .$ The duality mapping corresponding to the normalization
function $\varphi $ is the set-valued mapping $J:X\rightarrow 2^{X^{\ast }}$
given by
\begin{equation}
J\left( x\right) =\{x^{\ast }\in X^{\ast }\,:\,\langle x^{\ast
}\,,\,x\rangle =\varphi (|x|)|x|\,,\ \,|x^{\ast }|_{X^{\ast }}=\varphi
(|x|)\}.  \label{duality_mapping_definition}
\end{equation}%
Several fundamental properties of the duality mapping are summarized in the
following proposition. For proofs and additional information, we refer the
reader to \cite{ic,djm,d}.

\begin{proposition}
\label{pdm} The duality mapping \eqref{duality_mapping_definition} has the
following properties:

\begin{description}
\item[(a)] For each $x\in X,$ the set $J\left( x\right) $ is nonempty,
bounded, convex and closed;

\item[(b)] $J$ is monotone, i.e., $\text{for all }x,y\in X,\,x^{\ast }\in
J(x)\text{ and }y^{\ast }\in J(y),$ one has
\begin{equation*}
\langle x^{\ast }-y^{\ast },x-y\rangle \geq \left( \varphi \left( \left\vert
x\right\vert \right) -\varphi \left( \left\vert y\right\vert \right) \right)
\left( \left\vert x\right\vert -\left\vert y\right\vert \right) \geq 0;
\end{equation*}

\item[(c)] If $X$ is strictly convex, then $J$ is strictly monotone, i.e.,
for all $x,y\in X$ with $x\neq y,$ and $x^{\ast }\in J\left( x\right) ,$ $%
y^{\ast }\in J\left( y\right) ,$ one has
\begin{equation*}
\langle x^{\ast }-y^{\ast },x-y\rangle >0;
\end{equation*}

\item[(d)] If $X^{\ast }$ is strictly convex, then $J$ is single-valued;

\item[(e)] If $X$ is reflexive and $J$ is single-valued, then $J\left(
X\right) =X^{\ast }$ and $J$ is demicontinuous, i.e., if $x_{n}\rightarrow x$
strongly, then $J\left( x_{n}\right) \rightharpoonup J\left( x\right) $
weakly;

\item[(f)] If $X$ is reflexive and locally uniformly convex and $J$ is
single-valued, then $J$ is bijective from $X$ to $X^{\ast }$ and its inverse
$J^{-1}$ is bounded, continuous and monotone.
\end{description}
\end{proposition}

\subsection{Energetic Harnack inequality for the $p$-Laplacian}

As an example of duality mapping, we mention the case of the Sobolev space $%
W_{0}^{1,p}\left( 0,1\right)$ for $1<p<\infty$, endowed with the energetic
norm%
\begin{equation*}
\left\vert u\right\vert _{1,p}=\left( \int_{0}^{1}\left\vert u^{\prime
}\left( s\right) \right\vert ^{p}ds\right) ^{\frac{1}{p}},
\end{equation*}%
corresponding to the $p$-Laplace operator $J\left( u\right) =-\left(
\left\vert u^{\prime }\right\vert ^{p-2}u^{\prime }\right) ^{\prime }.$ This
operator is the duality mapping of the space $W_{0}^{1,p}\left( 0,1\right) $
corresponding to the normalization function $\varphi \left( \tau \right)
=\tau ^{p-1}.$ In virtue of the very good geometry of the space, $\ J$ is
invertible and its inverse
\begin{equation*}
J^{-1}:W^{-1,p^{\prime }}(0,1)\rightarrow W_{0}^{1,p}\left( 0,1\right), \ \
\ \left( \tfrac{1}{p}+\tfrac{1}{p^\prime}=1\right),
\end{equation*}%
is bounded, continuous, and strictly monotone.

With respect to the $p$-Laplace operator, we have the following Harnack type
inequality in terms of the energetic norm, obtained in \cite{lpv}. This
result proves to be extremely useful for localization in annular sets, as it
allows for obtaining certain lower bounds (see also \cite{ppv}).

\begin{lemma}[{\protect\cite[Lemma 3.1]{lpv}}]
\label{Harnack} For every function $u\in W_{0}^{1,p}(0,1)$ such that $%
u(t)\geq 0$ and $u(t)=u(1-t)$ for all $t\in \left( 0,\frac{1}{2}\right) $,
if $Ju\in C([0,1],\,\mathbb{R}_{+})$ and $Ju$ is nondecreasing on $\left( 0,%
\frac{1}{2}\right) $, then
\begin{equation}
u(t)\geq t(1-2t)^{1/(p-1)}|u|_{1,p},  \label{ih}
\end{equation}%
for all $t\in \left( 0,\frac{1}{2}\right) $.
\end{lemma}

\subsection{Birkhoff-Kellogg type theorem in cones}

One of our tools in this paper is the version in cones due to Krasnoselskii
and Ladyzenskii\ \cite{kl} (see also \cite[p.139]{ag}, \cite{guo} and \cite%
{cir}), of the classical invariant-direction theorem of Birkhoff and Kellogg
\cite{bk} (see also \cite{ag1}, \cite[Theorem 6.6]{ag}) regarding the
existence of a `nonlinear' eigenvalue and eigenvector for compact maps in
Banach spaces.

\begin{theorem}[Krasnoselskii and Ladyzenskii]
\label{tkl}Let $X$ be a real Banach space, $U\subset X$ be an open bounded
set with $0\in U$, $K\subset X$ a cone, and $T:K\cap \overline{U}\rightarrow
K$ a compact operator. If
\begin{equation*}
\inf_{x\in K\cap \partial U}\left\vert T\left( x\right) \right\vert >0,
\end{equation*}%
then, there exist $\lambda _{0}>0$ and $x_{0}\in K\cap \partial U$ such that
\begin{equation*}
x_{0}=\lambda _{0}T\left( x_{0}\right) .
\end{equation*}
\end{theorem}

\section{Main results}

In what follows, $X$ is Banach space with a single-valued strictly monotone
duality mapping $J,$ which is invertible, with continuous inverse $J^{-1};$ $%
K$ is a nondegenerate cone in $X,$ i.e., $K\subset X$ is closed, convex, $%
\lambda K\subset K$ for all $\lambda \in
\mathbb{R}
_{+}$ and $K\setminus \left\{ 0\right\} \neq \emptyset ;$ and $%
E:X\rightarrow \mathbb{R}$ is a $C^{1}$ Fr\'{e}chet differentiable
functional.

Given $0<r<R<\infty $, our aim is to determine a critical point of $E$
within the conical annular set
\begin{equation*}
K_{r,R}:=\left\{ u\in K\,:\,r\leq \left\vert u\right\vert \leq R\right\} .
\end{equation*}%
The Nehari manifold of the functional $E$ restricted to $K_{r,R}$ is the set
\begin{equation*}
\mathcal{N}_{r,R}=\left\{ u\in K_{r,R}\,:\,\langle E^{\prime }(u),u\rangle
=0\right\} .
\end{equation*}%
Obviously, $\mathcal{N}_{r,R}$ is a closed set.

In the subsequent, we consider the operators
\begin{equation*}
N\colon X\rightarrow X^{\ast },\quad N\left( u\right) :=J\left( u\right)
-E^{\prime }(u),
\end{equation*}%
and
\begin{equation*}
T:X\rightarrow X,\quad T=J^{-1}N.
\end{equation*}%
Our first two assumptions regard the operator $T$, and read as follows:

\begin{description}
\item[(H1)] The operator $T$ is completely continuous from $X$ to $X$, and
moreover, it is invariant with respect to the cone $K$, i.e.,
\begin{equation*}
T(K)\subset K.
\end{equation*}

\item[(H2)] The operator $T$ is bounded away from zero on the set $\mathcal{N%
}_{r,R}$, that is,
\begin{equation*}
\inf_{u\in \mathcal{N}_{r,R}}\left\vert T(u)\right\vert >0.
\end{equation*}
\end{description}

To state the third assumption, for each point $u\in K\setminus \{0\}$,
define the function
\begin{equation}
\alpha _{u}\colon \mathbb{R}_{+}\rightarrow \mathbb{R},\quad \alpha
_{u}(s )=E(s u).  \label{alpha}
\end{equation}%
Clearly,%
\begin{equation*}
\alpha _{u}^{\prime }\left(s \right) =\left\langle E^{\prime }\left(
s u\right) ,u\right\rangle \ \ \left( s \in
\mathbb{R}
_{+},\ u\in K\setminus \{0\}\right) .
\end{equation*}

\begin{description}
\item[(H3)] For every $u\in K_{r,R}$, one has
\begin{equation*}
\alpha _{u}^{\prime }\left( \frac{r}{|u|}\right) >0,\quad \alpha
_{u}^{\prime }\left( \frac{R}{|u|}\right) <0,
\end{equation*}%
and the equation
\begin{equation*}
\alpha _{u}^{\prime }(s )=0
\end{equation*}%
has a unique solution $s=:s(u)$ in the interval $\left[ \frac{r}{|u|},%
\frac{R}{|u|}\right] $.
\end{description}

Note that under condition (H3), for any $u\in K_{r,R},$ one has the strict
inequalities%
\begin{equation}
\frac{r}{|u|}<s\left( u\right) <\frac{R}{|u|}.  \label{f5}
\end{equation}

\begin{lemma}
\label{lc copy(1)}Under assumption (H3), the Nehari manifold $\mathcal{N}%
_{r,R}$ has the representations:
\begin{equation*}
\mathcal{N}_{r,R}=\left\{ \,s(u)\,u:\ u\in K_{r,R}\,\right\} =\left\{ u\in
K_{r,R}:\ s\left( u\right) =1\right\} .
\end{equation*}%
Moreover, for every $u\in \mathcal{N}_{r,R},$ one has
\begin{equation}
r<\left\vert u\right\vert <R.  \label{f2}
\end{equation}
\end{lemma}

\begin{proof}
Denote
\begin{equation*}
A:=\left\{ \,s(u)\,u:\ u\in K_{r,R}\,\right\} ,\ \ \ B:=\left\{ u\in
K_{r,R}:\ s\left( u\right) =1\right\} .
\end{equation*}

(a) If $u\in \mathcal{N}_{r,R},$ one has $u\in K_{r,R}$ and $\langle
E^{\prime }(u),u\rangle =0,$ that is $\alpha _{u}^{\prime }\left( 1\right)
=0.$ Hence $s\left( u\right) =1$ and so $u=s\left( u\right) u\in A.$ Thus $%
\mathcal{N}_{r,R}\subset A.$

(b) Conversely, if $u\in A,$ then $u=s\left( v\right) v$ for some $v\in
K_{r,R},$ and from the definition of $s,$
\begin{equation*}
0=\alpha _{v}^{\prime }\left( s\left( v\right) \right) =\left\langle
E^{\prime }\left( s\left( v\right) v\right),v \right\rangle =\left\langle
E^{\prime }\left( u\right) ,v\right\rangle ,
\end{equation*}%
whence $\left\langle E^{\prime }\left( u\right) ,u\right\rangle =0,$ showing
that $u\in \mathcal{N}_{r,R}.$ Thus $A\subset \mathcal{N}_{r,R}.$

Therefore $\mathcal{N}_{r,R}=A.$

(c) The inclusion $\mathcal{N}_{r,R}\subset B$ is clear in view of (a).

(d) If $u\in B,$ then $s\left( u\right) =1,$ so $\alpha _{u}^{\prime }\left(
1\right) =0,$ i.e., $\left\langle E^{\prime }\left( u\right) ,u\right\rangle
=0,$ which shows that $u\in \mathcal{N}_{r,R}.$ Thus $B\subset \mathcal{N}%
_{r,R}.$

Therefore $\mathcal{N}_{r,R}=B.$

To prove (\ref{f2}), take any $u\in \mathcal{N}_{r,R}.$ Then $s\left(
u\right) =1,$ and in virtue of (\ref{f5}), one has
\begin{equation*}
\frac{r}{\left\vert u\right\vert }<s\left( u\right) =1<\frac{R}{\left\vert
u\right\vert },
\end{equation*}%
which yields (\ref{f2}).
\end{proof}

We are now ready to state our first main result of this paper, whose proof
relies on the Krasnoselskii-Ladyzhenskii theorem.

\begin{theorem}
\label{t1} Assume that conditions (H1)-(H3) are satisfied. Then, there
exists $u^{\ast }\in \mathcal{N}_{r,R}$ such that $E^{\prime }\left( u^{\ast
}\right) =0.$
\end{theorem}

\begin{proof}
The central idea of the proof is to apply the Krasnoselskii-Ladyzenskii
theorem to the operator $T=J^{-1}N$ on an open bounded set $U$ chosen such
that
\begin{equation*}
K\cap \partial U=\mathcal{N}_{r,R}.
\end{equation*}%
To this end, we define the set
\begin{equation*}
V:=\left\{ \sigma u\,:\,u\in \mathcal{N}_{r,R},\,\ \sigma \in \lbrack
0,1)\right\} .
\end{equation*}%
Clearly, $0\in V$ and $V$ is bounded. Indeed, for any $w\in V\setminus \{0\}$
we have
\begin{equation*}
w=\sigma u,
\end{equation*}%
for some $u\in \mathcal{N}_{r,R}$ and $\sigma \in \left( 0,1\right) ,$ so
\begin{equation*}
|w|\leq |u|=|s(u)\,u|<\frac{R}{|u|}\,|u|=R.
\end{equation*}%
To verify that $V$ is open in the relative topology of $K$, it is sufficient
to show that $K\setminus V$ is closed. Let $w_{k}\in K\setminus V$ and $%
w_{k}\rightarrow w$. Then, there exist $u_{k}\in \mathcal{N}_{r,R}$ and $%
\sigma _{k}\geq 1$ such that
\begin{equation}
w_{k}=\sigma _{k}u_{k}.  \label{w_k}
\end{equation}%
Since the sequence $u_{k}$ is bounded (recall that $\left\vert
u_{k}\right\vert \in (r,R)$), it follows that the sequence $\sigma _{k}$ is
also bounded and therefore admits a convergent subsequence, which we denote
again by $\sigma _{k}$. Let $\sigma $ be the limit of this subsequence.
Clearly, $\sigma _{k}\geq 1$ implies that $\sigma \geq 1$. From \eqref{w_k},
we may write $u_{k}=\frac{1}{\sigma _{k}}w_{k}$, hence $u_{k}$ is convergent
to $u:=\frac{1}{\sigma }w$. As the Nehari manifold $\mathcal{N}_{r,R}$ is
closed, it follows that $u\in \mathcal{N}_{r,R}$. Consequently, the limit
point $w$ satisfies
\begin{equation*}
w=\sigma u\quad \text{with }\sigma \geq 1\text{ and }u\in \mathcal{N}_{r,R},
\end{equation*}%
which implies that $w\in K\setminus V$.

From the above argument, we also infer that
\begin{equation*}
\overline{V}=\left\{ \sigma u\,:\,u\in \mathcal{N}_{r,R},\,\ \sigma \in
\lbrack 0,1]\right\} ,
\end{equation*}%
hence the relative boundary (to $K$) of $V$ is given by
\begin{equation*}
\partial _{K}V:=\overline{V}\setminus V=\left\{ \sigma u\,:\,u\in \mathcal{N}%
_{r,R},\,\ \sigma =1\right\} =\mathcal{N}_{r,R}.
\end{equation*}%
To strictly comply with the conditions of Theorem \ref{tkl}, let us indicate
the open bounded set $U$ of the space $X$ to which it applies. Recall that,
by a theorem due to Dugundji (see \cite[Corollary~4.2]{jd}), every nonempty
closed convex subset of a real Banach space $X$ is a retract of $X$. In
particular, every cone of $X$ is a retract of $X$. Let $\rho \colon
X\rightarrow K$ be a retraction, i.e., a continuous mapping satisfying $\rho
(u)=u$ for all $u\in K$.

Since $V$ is bounded, we can choose $M>0$ such that $\overline{V}\subset
B_{M}$, where $B_{M}$ is the open ball of $X$ centered at the origin and of
radius $M.$ Then, the set
\begin{equation*}
U:=\rho ^{-1}(V)\cap B_{M}
\end{equation*}%
is open in $X$ and
\begin{equation*}
K\cap \partial U=\mathcal{N}_{r,R}.
\end{equation*}%
Indeed, using that $\rho (u)=u$ for all $u\in K$ and $\overline{V}\subset
K\cap B_{M}$, one has
\begin{equation*}
K\cap \overline{U}=K\cap \overline{\rho ^{-1}(V)\cap B_{M}}=\overline{V},
\end{equation*}%
whence
\begin{equation*}
K\cap \partial U=K\cap \partial (\rho ^{-1}(V)\cap B_{M})=\partial
_{K}\left( K\cap \rho ^{-1}(V)\cap B_{M})\right) =\partial _{K}V,
\end{equation*}%
which proves our claim.

Now, based on assumptions (H1) and (H2), Theorem \ref{tkl} applies to the
operator
\begin{equation*}
T\colon K\cap \overline{U}\rightarrow K
\end{equation*}%
and guarantees the existence of an element $u^{\ast }\in \mathcal{N}_{r,R}$
and a positive scalar $\lambda _{0}>0$ such that
\begin{equation*}
J^{-1}N\left( u^{\ast }\right) =\lambda _{0}u^{\ast },
\end{equation*}%
or, equivalently,
\begin{equation}
N\left( u^{\ast }\right) =J(\lambda _{0}u^{\ast }).  \label{f6}
\end{equation}%
By the definition of $\mathcal{N}_{r,R}$, we have
\begin{equation}
0=\langle J\left( u^{\ast }\right) -N\left( u^{\ast }\right) ,u^{\ast
}\rangle =\langle J\left( u^{\ast }\right) -J(\lambda _{0}u^{\ast }),u^{\ast
}\rangle ,  \label{ineq1}
\end{equation}%
which also implies that
\begin{equation}
0=\langle J\left( u^{\ast }\right) -N\left( u^{\ast }\right) ,u^{\ast
}\rangle =\langle J\left( u^{\ast }\right) -J(\lambda _{0}u^{\ast }),\lambda
_{0}u^{\ast }\rangle .  \label{ineq2}
\end{equation}%
Combining \eqref{ineq1} and \eqref{ineq2}, one derives
\begin{equation*}
\langle J\left( u^{\ast }\right) -J(\lambda _{0}u^{\ast }),u^{\ast }-\lambda
_{0}u^{\ast }\rangle =0.
\end{equation*}%
Using the strict monotonicity property of the duality mapping $J$, one
necessary have
\begin{equation*}
\lambda _{0}u^{\ast }=u^{\ast },
\end{equation*}%
hence $\lambda _{0}=1$. Consequently, equality (\ref{f6}) becomes $N\left(
u^{\ast }\right) =J(u^{\ast })$, therefore $u^{\ast }$ is a critical point
of the functional $E$.
\end{proof}

Let us note that hypotheses (H2) and (H3) require a certain behavior of the
functional $E$ only relative to an interval $[r,R]$. The situation in which
this behavior occurs on several such intervals leads us to the multiplicity
of critical points, with their location in disjoint annular conical sets.
Thus, Theorem~\ref{t1} directly yields the following multiplicity principle.

\begin{theorem}[Multiplicity]
\label{t2}Let condition (H1) hold.

\begin{description}
\item[(1$^{0}$)] If there are finite sequences of numbers $\left(
r_{k}\right) _{1\leq k\leq m}$ and $\left( R_{k}\right) _{1\leq k\leq m}$
with
\begin{equation*}
0<r_{1}<R_{1}<r_{2}<R_{2}<\dots <r_{m}<R_{m}
\end{equation*}%
such that conditions (H2) and (H3) are satisfied for every pair $\left(
r_{k},R_{k}\right) ,\ k=1,2,...,m,$ then there exist $m$ points $u_{k}^{\ast
}$ with%
\begin{equation*}
E^{\prime }\left( u_{k}^{\ast }\right) =0,\ \ u_{k}^{\ast }\in K,\ \
r_{k}<\left\vert u_{k}^{\ast }\right\vert <R_{k}\ \ \ \left(
k=1,2,...,m\right) .
\end{equation*}

\item[(2$^{0}$)] If there are increasing sequences of numbers $\left(
r_{k}\right) _{k\geq 1}$ and $\left( R_{k}\right) _{k\geq 1}$ with%
\begin{equation*}
0<r_{k}<R_{k}<r_{k+1}\ \ \left( k\geq 1\right) ,\ \ \ r_{k}\rightarrow
\infty \ \ \text{as }k\rightarrow \infty
\end{equation*}%
such that conditions (H2) and (H3) are satisfied for every pair $\left(
r_{k},R_{k}\right) ,\ k\geq 1,$ then there exists a sequence of points $%
\left( u_{k}^{\ast }\right) _{k\geq 1}$ with%
\begin{equation*}
E^{\prime }\left( u_{k}^{\ast }\right) =0,\ \ u_{k}^{\ast }\in K,\ \
r_{k}<\left\vert u_{k}^{\ast }\right\vert <R_{k},\ \ \left\vert u_{k}^{\ast
}\right\vert \rightarrow \infty \ \ \ \text{as }k\rightarrow \infty .
\end{equation*}

\item[(3$^{0}$)] If there are decreasing sequences of numbers $\left(
r_{k}\right) _{k\geq 1}$ and $\left( R_{k}\right) _{k\geq 1}$ with%
\begin{equation*}
0<R_{k+1}<r_{k}<R_{k}\ \ \left( k\geq 1\right) ,\ \ \ R_{k}\rightarrow 0\ \
\text{as }k\rightarrow \infty
\end{equation*}%
such that conditions (H2) and (H3) are satisfied for every pair $\left(
r_{k},R_{k}\right) ,\ k\geq 1,$ then there exists a sequence of points $%
\left( u_{k}^{\ast }\right) _{k\geq 1}$ with%
\begin{equation*}
E^{\prime }\left( u_{k}^{\ast }\right) =0,\ \ u_{k}^{\ast }\in K,\ \
r_{k}<\left\vert u_{k}^{\ast }\right\vert <R_{k},\ \ u_{k}^{\ast
}\rightarrow 0\ \ \ \text{as }k\rightarrow \infty .
\end{equation*}
\end{description}
\end{theorem}

\begin{proof}
The result follows directly by applying Theorem \ref{t1} to each pair $%
(r_{k},R_{k}).$
\end{proof}

\section{Application}

To illustrate the theoretical results, we consider the Dirichlet problem for
a $p$-Laplace equation
\begin{equation}
\begin{cases}
-\left( \left\vert u^{\prime }\right\vert ^{p-2}u^{\prime }\right) ^{\prime
}(t)=f(u(t)),\quad t\in \left( 0,1\right) \\
u(0)=u(1)=0,%
\end{cases}
\label{pa}
\end{equation}%
where $1<p<\infty $, and $f:\mathbb{R}\rightarrow \mathbb{R}$ is a
continuous function that is nonnegative and nondecreasing on $\mathbb{R}%
_{+}. $

Consider the Banach space $X:=W_{0}^{1,p}\left( 0,1\right) $ endowed with
the usual norm $|u|_{1,p}:=|u^{\prime }|_{p},$ where $\left\vert \cdot
\right\vert _{p}$ is the norm of $L^{p}\left( 0,1\right) $. It is well known
(see, e.g., \cite[Chapter~1.2]{djm}) that $W_{0}^{1,p}(0,1)$ is a uniformly
convex and reflexive Banach space with $W_{0}^{-1,p^{\prime }}(0,1)$ its
dual, where $p^{\prime }$ is the conjugate of $p$, i.e., $\frac{1}{p}+\frac{1%
}{p^{\prime }}=1$. If $\langle \cdot ,\cdot \rangle $ denotes the duality
pairing between $W_{0}^{-1,p^{\prime }}(0,1)$ and $W_{0}^{1,p}(0,1),$ and $%
v\in L^{q}(0,1)\subset W_{0}^{-1,p^{\prime }}(0,1)$, then
\begin{equation*}
\langle v,u\rangle =\int_{0}^{1}v(t)u(t)dt,
\end{equation*}%
for all $u\in W_{0}^{1,p}(0,1)$ (see \cite[Proposition~8.14]{b}).

The duality mapping of $W_{0}^{1,p}\left( 0,1\right) $ is%
\begin{equation*}
J\left( u\right) =-\left( \left\vert u^{\prime }\right\vert ^{p-2}u^{\prime
}\right) ^{\prime }.
\end{equation*}

Let $\lambda _{p}$ denote the first eigenvalue of the Euler-Lagrange
equation
\begin{equation*}
J\left( u\right) =\lambda |u|_{1,p}^{p-2}u\quad \text{in }(0,1),\quad
u(0)=u(1)=0.
\end{equation*}%
It is known (see, e.g., \cite{djm}) that
\begin{equation*}
\lambda _{p}=\min_{u\in W_{0}^{1,p}\left( 0,1\right) \setminus \{0\}}\frac{%
|u|_{1,p}^{p}}{|u|_{p}^{p}},
\end{equation*}%
that is, $c_{p}=\lambda _{p}^{-\frac{1}{p}}$ is the smallest constant such
that
\begin{equation}
\left\vert u\right\vert _{p}\leq c_{p}|u|_{1,p},  \label{f8}
\end{equation}%
for all $u\in W_{0}^{1,p}(0,1)$. Also, for the continuous embedding $%
W_{0}^{1,p}(0,1)\subset C[0,1]$ one has
\begin{equation*}
\left\vert u(t)\right\vert \leq |u|_{1,p}\ \ \ \left( t\in \left[ 0,1\right]
\right) ,
\end{equation*}%
for every $u\in W_{0}^{1,p}(0,1)$.


The energy functional of the problem \eqref{pa} is
\begin{equation*}
E(u)=\frac{1}{p}\left\vert u\right\vert _{1,p}^{p}-\int_{0}^{1}F(u(t))dt,
\end{equation*}%
where $F(\xi )=\int_{0}^{\xi }f(s)ds$. One has
\begin{equation*}
E^{\prime }(u)=J\left( u\right) -N_{f}\left( u\right) ,
\end{equation*}%
where $N_{f}$ is the Nemytskii superposition operator $N_{f}\left( u\right)
:=f(u).$ Hence the solutions of problem (\ref{pa}) are the critical points
of the functional $E.$

In $W_{0}^{1,p}(0,1)$, we consider the cone
\begin{equation*}
\begin{aligned} K = \Big\{ u \in W_0^{1,p}(0,1) \, : \, &\  u \geq 0, \,u
\text{ is nondecreasing on } [0,1/2], \\ & u(t) = u(1-t)\text{ and } u(t)
\geq \phi(t) |u|_{1,p} \text{ for all } t \in [0,1/2] \Big\}, \end{aligned}
\end{equation*}%
where $\phi $ is the function involved in the energetic Harnack inequality (%
\ref{ih}), namely%
\begin{equation*}
\phi :\left[ 0,\frac{1}{2}\right] \rightarrow \mathbb{R},\text{ \quad $\phi
(t)=t(1-2t)^{1/(p-1)}$.}
\end{equation*}

%
%

Let $\beta \in (0,1/4)$ be fixed, and denote
\begin{equation*}
\Phi :=\int_{\beta }^{1/2}\phi (t)dt.
\end{equation*}%
For $0<r<R<\infty $, we assume that the following conditions hold:

\begin{description}
\item[(h1)] The function $f$ satisfies the inequalities
\begin{equation}  \label{conditie r}
\frac{f(r)}{r^{p-1}}<\frac{1}{c_{p}}\quad \text{ and }\quad \frac{f(R\phi
(\beta ))}{ R^{p-1}}>\frac{1}{2\Phi }.
\end{equation}

\item[(h2)] The function
\begin{equation*}
g\left( t\right) :=\frac{f(t)}{t^{p-1}}\,\text{ }
\end{equation*}%
is strictly increasing on $(0,R].$
\end{description}

Given the above two conditions, the following existence result holds.

\begin{theorem}
\label{pa1} Under conditions (h1) and (h2), problem~\eqref{pa} admits a
solution $u\in K$ such that
\begin{equation*}
r<|u|_{1,p}<R.
\end{equation*}
\end{theorem}

\begin{proof}
We apply Theorem \ref{t1}.

\textit{Check of }(H1)\textit{.}

(a)\textit{Complete continuity of the operator $J^{-1}N_{f}$.}The operator
$J^{-1}N_{f}$ is completely continuous from $W_{0}^{1,p}(0,1)$ to itself.
Indeed, since $W_{0}^{1,p}(0,1)$ compactly embeds into $C\left[ 0,1\right] $
(see, e.g., \cite[Theorem 8.8]{b}), and $C\left[ 0,1\right] $ continuously
embeds into $W^{-1,p^{\prime }}(0,1)$, the continuity of $f$ ensures that
the Nemytskii operator $N_{f}$ is completely continuous from $%
W_{0}^{1,p}(0,1)$ to $W^{-1,p^{\prime }}(0,1)$. Finally, since $J^{-1}$ is a
continuous bounded operator, it follows that $J^{-1}N_{f}$ is completely
continuous, as claimed.

(b) \textit{Invariance of the cone. }
We next show that the cone \(K\) is invariant under the operator \(J^{-1}N_{f}\), i.e., $J^{-1}N_{f}(K)\subset K$.
Let \(u \in K\) and set \(v = J^{-1}N_{f}(u)\).
We claim that \(v \in K\).

By the comparison principle for the $p$-Laplace operator (see, e.g., \cite[%
Lemma~1.3]{ac}), since $N_{f}\left( u\right) \geq 0$, it follows that $%
J^{-1}N_{f}\left( u\right) \geq 0$. To prove that $v$ is symmetric, denote $%
w\left( t\right) :=v\left( 1-t\right) .$ Since $u$ is symmetric, so is $f(u)$%
, and hence
\begin{equation*}
N_{f}(u(1-t))=N_{f}(u(t)).
\end{equation*}%
Moreover, we have
\begin{equation*}
J(v)(1-t)=N_{f}(u(1-t)),
\end{equation*}%
and
\begin{equation*}
J(v)(1-t)=J(w)(t).
\end{equation*}%
Therefore, 
\begin{equation*}
J\left( w\right) =N_{f}\left( u\right) ,
\end{equation*}%
which shows that both $v$ and $w$ solve the same Dirichlet problem for the $%
p $-Laplace equation. By uniqueness of the solution to this problem, it
follows that $v=w$, that is, $v$ is symmetric. Finally, we observe that $%
J\left( v\right) =J(J^{-1}N_{f}\left( u\right) )=N_{f}\left( u\right) $,
which is nonnegative and nondecreasing on $[0,1/2]$. Therefore, Lemma \ref%
{Harnack} applies and guarantees that the inequality \eqref{ih} holds for $v$%
. Consequently, $v\in K$ as claimed. Hence condition (H1) holds.

\textit{Check of }(H2)\textit{.} Suppose, by contradiction, that (H2) does
not hold. Then, one can find a sequence $(u_{k})\subset \mathcal{N}_{r,R}$
such that
\begin{equation}
J^{-1}N_{f}\left( u_{k}\right) \rightarrow 0.
\label{convergenta la zero J minus 1}
\end{equation}%
Since $J$ is demicontinuous (Proposition \ref{pdm} (e)), relation
\eqref{convergenta la zero
J minus 1} implies that
\begin{equation*}
J(J^{-1}N_{f}\left( u_{k}\right) )=N_{f}\left( u_{k}\right) \rightharpoonup
0,
\end{equation*}%
weakly. Therefore, for any given $\chi \in K\setminus \{0\}$, we have
\begin{equation*}
\langle N_{f}\left( u_{k}\right) ,\chi \rangle \rightarrow 0.
\end{equation*}%
From the Harnack inequality and the monotonicity of the functions $u_{k}$
and $\chi $ on the interval $\left[ 0,1/2\right] ,$ for every $t\in (\beta
,1/2),$ we have
\begin{align}
& u_{k}(t)\geq u_{k}\left( \beta \right) \geq \phi (\beta )|u_{k}|_{1,p},\,%
\text{ and }  \notag \\
& \chi (t)\geq \phi (t)|\chi |_{1,p}.  \label{ineq-uk}
\end{align}%
Using the monotonicity of $f$, the symmetry of $u_{k}$ and $\chi $, and the
bounds in \eqref{ineq-uk}, we obtain
\begin{align*}
\langle N_{f}\left( u_{k}\right) ,\chi \rangle &
=\int_{0}^{1}f(u_{k}(t))\chi (t)\,dt\geq 2\int_{\beta }^{1/2}f\left(
u_{k}(t)\right) \chi (t)\,dt \\
& \geq 2|\chi |_{1,p}f(\phi (\beta )|u_{k}|_{1,p})\int_{\beta }^{1/2}\phi
(t)\,dt=2|\chi |_{1,p}f(\phi (\beta )|u_{k}|_{1,p})\Phi .
\end{align*}%
Since $\left\vert u_{k}\right\vert _{1,p}\geq r$ for all $k,$ we conclude
that
\begin{equation*}
0\leq 2|\chi |_{1,p}f(\phi (\beta )r)\Phi \leq \langle N_{f}\left(
u_{k}\right) ,\chi \rangle \rightarrow 0,
\end{equation*}%
whence $f(\phi (\beta )r)=0,$ which contradicts the strict positivity of $f$
on $(0,R]$ implied by (h2). Hence, (H2) holds.

\textit{Check of }(H3)\textit{.} Let $u\in K_{r,R}$ and denote $w:=\frac{u}{%
|u|_{1,p}}.$ We immediately see that the derivative of the mapping $\alpha
_{u}$ defined in \eqref{alpha} is
\begin{equation*}
\alpha _{u}^{\prime }\left( \sigma \right) =\sigma
^{p-1}|u|_{1,p}^{p}-\int_{0}^{1}f(\sigma u(t))u(t)\,dt.
\end{equation*}%
We claim that
\begin{equation}
\alpha _{u}^{\prime }\left( \frac{r}{|u|_{1,p}}\right) >0\quad \text{and}%
\quad \alpha _{u}^{\prime }\left( \frac{R}{|u|_{1,p}}\right) <0.
\label{capete}
\end{equation}%
Since $w(t)\leq 1$ for all $t\in \lbrack 0,1]$, it follows that
\begin{equation}
f(rw(t))\leq f(r)\ \ \ \text{for all\ }t\in \left[ 0,1\right] .  \label{f7}
\end{equation}%
Moreover, by H\"{o}lder's inequality and (\ref{f8}), we have
\begin{equation}
\int_{0}^{1}w(t)\,dt\leq \left( \int_{0}^{1}w(t)^{p}\,dt\right) ^{1/p}\leq
c_{p}|w|_{1,p}=c_{p},  \label{f9}
\end{equation}%
since $|w|_{1,p}=1$ by definition. Thus, using (\ref{f7}), (\ref{f9}) and
the first inequality in \eqref{conditie r}, we obtain
\begin{align*}
\alpha _{u}^{\prime }\left( \frac{r}{|u|_{1,p}}\right) &
=r^{p-1}|u|_{1,p}-\int_{0}^{1}f(rw(t))u(t)dt \\
& =|u|_{1,p}\left( r^{p-1}-\int_{0}^{1}f(rw(t))w(t)dt\right) \\
& \geq |u|_{1,p}\left( r^{p-1}-f(r)\int_{0}^{1}w(t)dt\right) \\
& \geq |u|_{1,p}\left( r^{p-1}-c_{p}f(r)\right) \\
& >0,
\end{align*}%
that is, the first inequality in (\ref{capete}). To prove the second
inequality in (\ref{capete}), note that the monotonicity of $w$ on $[0,1/2]$%
, together with the Harnack inequality, yields that
\begin{equation}
w(t)\geq \phi (\beta )\quad \text{ for all $t\in \lbrack \beta ,1/2]$}.
\label{inegalitate w}
\end{equation}%
Using the symmetry of $w$ and the second inequality in \eqref{conditie r},
we find that
\begin{align*}
\alpha _{u}^{\prime }\left( \frac{R}{|u|_{1,p}}\right) &
=R^{p-1}|u|_{1,p}-\int_{0}^{1}f(Rw(t))u(t)dt \\
& =|u|_{1,p}\left( R^{p-1}-2\int_{0}^{1/2}f(Rw(t))w(t)dt\right) \\
& \leq |u|_{1,p}\left( R^{p-1}-2\int_{\beta }^{1/2}f(R\phi (t))\phi
(t)dt\right) \\
& \leq |u|_{1,p}\left( R^{p-1}-2\Phi f(R\phi (\beta ))\right) \\
& <0.
\end{align*}%
Consequently, the second inequality in \eqref{capete} also holds.

To continue with the verification of (H3), let us denote $w:=\frac{u}{%
|u|_{1,p}}$ and $\gamma :=\sigma \left\vert u\right\vert _{1,p}$. Then,
\begin{align*}
\alpha _{u}^{\prime }(\sigma )& =\sigma ^{p-1}\left\vert u\right\vert
_{1,p}^{p}-\int_{0}^{1}f\left( \sigma u\left( t\right) \right) u\left(
t\right) dt \\
& =\sigma ^{p-1}\left\vert u\right\vert _{1,p}^{p}\left( 1-\int_{0}^{1}\frac{%
f\left( \gamma w\left( t\right) \right) }{\gamma ^{p-1}}w\left( t\right)
dt\right) \\
& =\sigma ^{p-1}|u|_{1,p}^{p}h(\gamma ),
\end{align*}%
where
\begin{equation*}
h(\gamma ):=1-\int_{0}^{1}\frac{f(\gamma w(t))}{\gamma ^{p-1}}w(t)dt\ \ \
\left( \gamma \in \left[ r,R\right] \right) .
\end{equation*}%
We now show that the function $h$ is strictly decreasing on $\left[ r,R%
\right] .$ For this, let $r\leq \gamma _{1}<\gamma _{2}\leq R.$ One has%
\begin{equation*}
h(\gamma _{1})-h(\gamma _{2})=2\int_{0}^{\frac{1}{2}}\left( \frac{f(\gamma
_{2}w(t))}{\gamma _{2}^{p-1}}-\frac{f(\gamma _{1}w(t))}{\gamma _{1}^{p-1}}%
\right) w\left( t\right) dt.
\end{equation*}%
%
%
%
%
%
%
Since $0<w(t)\leq 1$ for all $t\in (0,1/2]$, it follows that
\begin{equation*}
0<\gamma _{1}w(t)<\gamma _{2}w(t)\leq R.
\end{equation*}%
Then, using assumption (h2), we have
\begin{equation*}
g(\gamma _{2}w(t))>g(\gamma _{1}w(t)),
\end{equation*}%
which implies
\begin{align*}
\left( \frac{f(\gamma _{2}w(t))}{\gamma _{2}^{p-1}}-\frac{f(\gamma _{1}w(t))%
}{\gamma _{1}^{p-1}}\right) w(t)& =\left( \frac{f(\gamma _{2}w(t))}{\gamma
_{2}^{p-1}w(t)^{p-1}}-\frac{f(\gamma _{1}w(t))}{\gamma _{1}^{p-1}w(t)^{p-1}}%
\right) w(t)^{p} \\
& =\left( g(\gamma _{2}w(t))-g(\gamma _{1}w(t))\right) w(t)^{p} \\
& >0,
\end{align*}%
for all $t\in (0,1/2]$. Therefore, $h\left( \gamma _{1}\right) >h\left(
\gamma _{2}\right) $ and thus $h$ is strictly decreasing on $\left[ r,R%
\right] .$ Moreover, since $h\left( r\right) >0$ and $h\left( R\right) <0,$
it follows that $h$ has exactly one zero $\gamma _{0}$ in $\left( r,R\right)
.$

Consequently, the derivative $\alpha _{u}^{\prime }$ has a unique zero in $%
\left[ \frac{r}{\left\vert u\right\vert _{1,p}},\frac{R}{\left\vert
u\right\vert _{1,p}}\right] $, namely
\begin{equation*}
s(u)=\frac{\gamma _{0}}{|u|_{1,p}},
\end{equation*}
so condition (H3) is satisfied.

Since all the conditions (H1)-(H3) are satisfied, Theorem \ref{t1} applies
and gives the conclusion.
\end{proof}

Instead of condition (h2), we may consider an alternative assumption
formulated in relation to the annular conical set $K_{r,R}$. More exactly,

\begin{enumerate}
\item[\textbf{(h2')}] The function $f$ is of class $C^{1}$ on $\mathbb{R}$
and
\begin{equation}
\min_{t\in \lbrack r\phi (\beta ),R]}f^{\prime }(t)>\frac{(p-1)R^{p-2}}{%
2\Psi }\text{ },  \label{ider}
\end{equation}%
where
\begin{equation*}
\Psi =\int_{\beta }^{1/2}\phi (t)^{2}dt.
\end{equation*}
\end{enumerate}

\begin{theorem}
Under conditions (h1) and (h2'), the problem (\ref{pa}) has a solution $u\in
K$ satisfying%
\begin{equation*}
r<\left\vert u\right\vert _{1,p}<R.
\end{equation*}
\end{theorem}

\begin{proof}
Similar to the proof of Theorem \ref{pa1}, assumption (H1) from Theorem \ref%
{t1} is satisfied. In addition, it is easy to see that the strict positivity
of $f^{\prime }\left( r\phi \left( \beta \right) \right) $ implies $f\left(
r\phi \left( \beta \right) \right) >0,$ which as above guarantees (H2).
Moreover, by (h1), relation \eqref{capete} also holds. To verify condition
(H3), it remains to prove that $\alpha _{u}^{\prime }$ has a unique zero $%
s\left( u\right) $ within the interval $\left[ \frac{r}{|u|_{1,p}},\frac{R}{%
|u|_{1,p}}\right] .$ Under the notations from the proof of the previous
theorem, consider the function
\begin{equation*}
\Tilde{h}(\gamma ):=\gamma ^{p-1}-\int_{0}^{1}f(\gamma w(t))w(t)dt=\gamma
^{p-1}h\left( \gamma \right) \ \ \ \left( \gamma \in \left[ r,R\right]
\right) .
\end{equation*}%
Then, we have
\begin{equation}
\alpha _{u}^{\prime }(\sigma )=|u|_{1,p}\Tilde{h}(\sigma |u|_{1,p})=|u|_{1,p}%
\Tilde{h}(\gamma ).  \label{alpha_prim_u_aplicatie}
\end{equation}%
We now show that $\Tilde{h}$ is strictly decreasing on $\left[ r,R\right] $.
Since $f$ is of class $C^{1}$, it suffices to show that
\begin{equation*}
\Tilde{h}^{\prime }(\gamma )<0,\quad \text{for all }\gamma \in \lbrack r,R].
\end{equation*}%
Differentiating, and using that $f$ is nondecreasing, we obtain
\begin{align*}
\Tilde{h}^{\prime }\left( \gamma \right) & =\left( p-1\right) \gamma
^{p-2}-2\int_{0}^{1/2}f^{\prime }(\gamma w\left( t\right) )w\left( t\right)
^{2}dt \\
& \leq \left( p-1\right) \gamma ^{p-2}-2\int_{\beta }^{1/2}f^{\prime
}(\gamma w\left( t\right) )w\left( t\right) ^{2}dt.
\end{align*}%
For $t\in \left[ \beta ,1/2\right] ,$ one has $r\phi \left( \beta \right)
\leq \gamma w\left( t\right) \leq R,$ whence%
\begin{equation*}
f^{\prime }(\gamma w\left( t\right) )\geq \min_{s\in \left[ r\phi \left(
\beta \right) ,R\right] }f^{\prime }\left( s\right) .
\end{equation*}%
Then, also using (\ref{ider}), we obtain
\begin{equation*}
\Tilde{h}^{\prime }\left( \gamma \right) \leq \left( p-1\right)
R^{p-2}-2\Psi \min_{s\in \left[ r\phi \left( \beta \right) ,R\right]
}f^{\prime }\left( s\right) <0,
\end{equation*}%
as we desired. 
Finally, by \eqref{capete} and \eqref{alpha_prim_u_aplicatie}, we conclude
that $\alpha _{u}^{\prime }$ has a unique zero within the interval $\left[
\frac{r}{|u|_{1,p}},\frac{R}{|u|_{1,p}}\right] .$ So condition (H3) is
verified.

Therefore, Theorem \ref{t1} applies and gives the conclusion.
\end{proof}

\begin{remark}
Condition (h2) given on the whole interval $(0,R]$ does not lead to
multiplicity. To see why, suppose there are two pairs $(r_{i},R_{i})$, $%
i=1,2 $, with $0<r_{1}<R_{1}<r_{2}<R_{2}$ for which conditions (h1) and (h2)
are both satisfied. By assumption (h1), we have
\begin{equation*}
\frac{f(r_{2})}{r_{2}^{p-1}}<\frac{1}{c_{p}}.
\end{equation*}%
On the other hand, (h2) yields
\begin{equation*}
\frac{f(R_{1})}{R_{1}^{p-1}}\leq \frac{f(r_{2})}{r_{2}^{p-1}}<\frac{1}{c_{p}}%
.
\end{equation*}%
This implies
\begin{align*}
\alpha _{u}^{\prime }\left( \frac{R_{1}}{|u|_{1,p}}\right) &
=R_{1}^{p-1}|u|_{1,p}-\int_{0}^{1}f(R_{1}w(t))u(t)dt \\
& =|u|_{1,p}\left( R_{1}^{p-1}-\int_{0}^{1}f(R_{1}w(t))w(t)dt\right) \\
& \geq |u|_{1,p}\left( R_{1}^{p-1}-c_{p}f(R_{1})\right) \\
& >0,
\end{align*}%
which contradicts \eqref{capete}. However, condition (h2') can be applied
separately to each of several disjoint annular sets, which leads to multiple
solutions, as shown by the following result illustrating the general
multiplicity principle given by Theorem~\ref{t2}
\end{remark}

\begin{theorem}

\begin{description}
\item[(1$^{0}$)] If there are finite sequences of numbers $\left(
r_{k}\right) _{1\leq k\leq m}$ and $\left( R_{k}\right) _{1\leq k\leq m}$
with
\begin{equation*}
0<r_{1}<R_{1}<r_{2}<R_{2}<\dots <r_{m}<R_{m}
\end{equation*}%
such that conditions (h1) and (h2') are satisfied for every pair $\left(
r_{k},R_{k}\right) ,\ k=1,2,...,m,$ then there exist $m$ solutions $%
u_{k}^{\ast }$ of problem (\ref{pa}) with%
\begin{equation*}
u_{k}^{\ast }\in K,\ \ r_{k}<\left\vert u_{k}^{\ast }\right\vert
_{1,p}<R_{k}\ \ \ \left( k=1,2,...,m\right) .
\end{equation*}

\item[(2$^{0}$)] If there are increasing sequences of numbers $\left(
r_{k}\right) _{k\geq 1}$ and $\left( R_{k}\right) _{k\geq 1}$ with%
\begin{equation*}
0<r_{k}<R_{k}<r_{k+1}\ \ \left( k\geq 1\right) ,\ \ \ r_{k}\rightarrow
\infty \ \ \text{as }k\rightarrow \infty
\end{equation*}%
such that conditions (h1) and (h2') are satisfied for every pair $\left(
r_{k},R_{k}\right) ,\ k\geq 1,$ then there exists a sequence of solutions $%
\left( u_{k}^{\ast }\right) _{k\geq 1}$ of problem (\ref{pa}) with%
\begin{equation*}
u_{k}^{\ast }\in K,\ \ r_{k}<\left\vert u_{k}^{\ast }\right\vert
_{1,p}<R_{k},\ \ \left\vert u_{k}^{\ast }\right\vert _{1,p}\rightarrow
\infty \ \ \ \text{as }k\rightarrow \infty .
\end{equation*}

\item[(3$^{0}$)] If there are decreasing sequences of numbers $\left(
r_{k}\right) _{k\geq 1}$ and $\left( R_{k}\right) _{k\geq 1}$ with%
\begin{equation*}
0<R_{k+1}<r_{k}<R_{k}\ \ \left( k\geq 1\right) ,\ \ \ R_{k}\rightarrow 0\ \
\text{as }k\rightarrow \infty
\end{equation*}%
such that conditions (h1) and (h2') are satisfied for every pair $\left(
r_{k},R_{k}\right) ,\ k\geq 1,$ then there exists a sequence of solutions $%
\left( u_{k}^{\ast }\right) _{k\geq 1}$ of problem (\ref{pa}) with%
\begin{equation*}
u_{k}^{\ast }\in K,\ \ r_{k}<\left\vert u_{k}^{\ast }\right\vert
_{1,p}<R_{k},\ \ u_{k}^{\ast }\rightarrow 0\ \ \ \text{as }k\rightarrow
\infty .
\end{equation*}
\end{description}
\end{theorem}

\begin{remark}
It is worth mentioning that requirement (h1) can be satisfied for a sequence
of pairs $(r_{k},R_{k})$ as in (2$^{0}$), if for example%
\begin{equation*}
\lim \inf_{\tau \rightarrow \infty }\frac{f\left( \tau \right) }{\tau ^{p-1}}%
<\frac{1}{c_{p}}\ \ \ \text{and\ \ \ }\lim \sup_{\tau \rightarrow \infty }%
\frac{f\left( \tau \right) }{\tau ^{p-1}}>\frac{1}{2\Phi \phi \left( \beta
\right) ^{p-1}}.
\end{equation*}%
Similarly, (h1) can be satisfied for a sequence of pairs $(r_{k},R_{k})$ as
in (3$^{0}$), if
\begin{equation*}
\lim \inf_{\tau \rightarrow 0+}\frac{f\left( \tau \right) }{\tau ^{p-1}}<%
\frac{1}{c_{p}}\ \ \ \text{and\ \ \ }\lim \sup_{\tau \rightarrow 0+}\frac{%
f\left( \tau \right) }{\tau ^{p-1}}>\frac{1}{2\Phi \phi \left( \beta \right)
^{p-1}}.
\end{equation*}%
Both situations mean a very strong oscillation towards infinity and zero,
respectively, from values smaller than $\frac{1}{c_{p}}$ to values larger
than $\frac{1}{2\Phi \phi \left( \beta \right) ^{p-1}}.$
\end{remark}
\subsection*{Acknowledgments}

The authors are grateful to the reviewer for the careful reading of the
manuscript and for the valuable remarks that helped improve the quality of
the paper.



\end{document}